 \documentclass[a4paper,12pt,reqno]{amsart}
\usepackage{amsmath}%
\usepackage{amsfonts}%
\usepackage{amssymb}%
\usepackage{graphicx}
\newtheorem{theorem}{Theorem}
\theoremstyle{plain}

\newtheorem{corollary}{Corollary}
\newtheorem{definition}{Definition}
\newtheorem{lemma}{Lemma}
\newtheorem{remark}{Remark}
\numberwithin{equation}{section}
\begin{document}
\title[Infinite-dimensional Compact Quantum Semigroup]{Infinite-dimensional Compact Quantum Semigroup}
\author{M.A.Aukhadiev}
\address[M.A.Aukhadiev, S.A.Grigoryan and E.V.Lipacheva]
{Kazan State Energetics University, Krasnoselskaya str., 51, 420066, Kazan, Russia }
\email[M.A.Aukhadiev]{m.aukhadiev@gmail.com}%
\author{S.A.Grigoryan}
\email[S.A.Grigoryan]{gsuren@inbox.ru}%
\author{E.V.Lipacheva}
\email[E.V.Lipacheva]{elipacheva@gmail.com}%
\date{April 26, 2011}
\dedicatory{Dedicated to Professor S.L. Woronowicz.}

\begin{abstract}
In this paper we construct a compact quantum semigroup structure on the Toeplitz algebra $\mathcal{T}$. The existence of a subalgebra, isomorphic to the algebra of regular Borel's measures on a circle with convolution product, in the dual algebra $\mathcal{T}^*$ is shown. The existence of Haar functionals in the dual algebra and in the above-mentioned subalgebra is proved. Also we show the connection between $\mathcal{T}$ and the structure of weak Hopf algebra.  \end{abstract}\maketitle

\section{Introduction}
\label{intro}
By the name \emph{quantum group} V.G. Drinfeld introduced a new class of Hopf algebras, which provides the solution of a well-known Yang-Baxter equation.
The standart exapmle of the notion of the Hopf algebra is the commutative algebra $A$ of functions on a unimodular group $G$. The group operation induces the algebraic homomorphism $\Delta\colon A\to A$, called \emph{comultiplication}, which takes each function $f(x)$ to the function of two arguments.
$$\Delta(f)(x,y)=f(xy)$$
The existence of the inverse element can be encoded in an anti-isomorphism $S$ on $A$. If $\Delta$ and $S$ satisfy some specific conditions, then $(A,\Delta,S)$ is called a Hopf algebra. 

Besides purely algebraic works in the theory of quantum groups soon appeared those, where topology played an important role. The theory of quantum groups in the framework of C*-algebras starts in 1990s with the works of S.L.Woronowicz, the establisher of this approach. The definition of a \emph{compact quantum group} (CQG), proposed by Woronowicz, was general enough to contain the newly-discovered CQGs -- $SU_q(2)$ and the algebra of continuous functions on compact group. 

The natural generalization of this notion is a compact quantum semigroup. The definition of it appears for the first time in the work of Alfons Van Daele and Ann Maes \cite{Vandaele} in 1998. The common example of this notion is an algebra of continuous functions on a compact semigroup. Later it turns out that some compact quantum semigroups possess interesting properties. 

 The aim of this paper is to provide a non-trivial example of an infinite-dimensional C*-algebra, which would admit the compact quantum semigroup structure, and to study relation between the structures of compact quantum semigroups and compact quantum groups by means of this example. The main object of this work is the Toeplitz algebra -- the C*-algebra generated by an isometrical unilateral shift operator $T$, and uniformly closed.
 
  The paper is organized as follows. In Section~\ref{sec:pre} we set up notations and recall the basic facts concerning the Toeplitz algebra $\mathcal{T}$. We study the structure of $\mathcal{T}$ in Section~\ref{sec:str}. We discover that each element of this algebra can be represented by the infinite series, and this expansion corresponds to the way of grading for the algebra. The characteristics of this series allow us to give a rigorous criterion of compactness for operators in this algebra, which is shown in Section~\ref{sec:comp}.
 
 In Section~\ref{sec:q} we construct compact quantum semigroup structure on the Toeplitz algebra $\mathcal{T}$. Firstly, we define the comultiplication $\Delta$ on the monoid generated by $T$ and $T^*$ in the most natural way. Since this monoid is an inverse semigroup \cite{Clifford}, then the algebra generated by this monoid becomes a Weak Hopf Algebra \cite{Li} with operation $\Delta$. On the other hand, the same comultiplication leads us to the cocommutative compact quantum semigroup structure on the whole Toeplitz algebra. Thus, the two stuctures are closely related to each other in this example.

The comultiplication on the Toeplitz algebra naturally induces the Banach algebra structure on the dual algebra, as shown in Section~\ref{sec:dual}. Moreover, this algebra is commutative and unital, where the role of unit plays a counit for the coalgebra $(\mathcal{T},\Delta)$. Futher we find the Haar functional $h$, defined in \cite{Woronowicz}. It is known that not every compact quantum semigroup admits a Haar functional \cite{Bedos}. We show that this functional $h$ can be obtained in the same way as in \cite{Woronowicz}, despite the fact that $(\mathcal{T},\Delta)$ is not a compact quantum group.

Let $\mathcal{K}$ be the algebra of compact operators. It is well-known that $\mathcal{K}$ is the ideal in $\mathcal{T}$. In Section~\ref{sec:dual} we show that the algebra $(\mathcal{T}/\mathcal{K})^*$ is isomorphic to the algebra of Borel measures on the unit circle with convolution product. Also we prove that there exists the Haar functional in $(\mathcal{T}/\mathcal{K})^*$, which is different from $h$. Thus, we may conclude that $(\mathcal{T}/\mathcal{K})^*$ is the dual algebra to the compact quantum group $(C(S^1),\Delta)$.

\section{Preliminaries}
\label{sec:pre}
Let $l^2(\mathbb{Z}_{+})$ be the Hilbert space of all complex-valued functions on $\mathbb{Z}_+$ satisfying $$f\colon \mathbb{Z}_{+}\to\mathbb{C},\ \sum_{n=0}^\infty \mid f(n)\mid^2<\infty.$$
The set of functions $\{e_n\}_{n=0}^\infty,\ e_n(m)=\delta_{n,m}$ forms an orthonormal basis in $l^2(\mathbb{Z}_{+})$, where $\delta_{n,m}$ denotes the Kronecker symbol. A bounded linear operator $T\colon l^2(\mathbb{Z}_{+})\to l^2(\mathbb{Z}_{+})$ is called a right shift operator if 
$$Te_n=e_{n+1},\ for\ all\ n\in \mathbb{Z}_{+}.$$
Obviously $T$ is isometric. And $T^*$ is a left shift operator:
$$T^*e_0=0,$$
$$T^*e_n=e_{n-1}, n>0.$$
Denote by $\mathcal{B}(l^2(\mathbb{Z}_{+}))$ the algebra of all bounded linear operators on $l^2(\mathbb{Z}_{+})$. Toeplitz algebra is a uniformly closed subalgebra in $\mathcal{B}(l^2(\mathbb{Z}_{+}))$ generated by $T$ and $T^*$. Denote this algebra by $\mathcal{T}$.  
\par $TT^*$ is a projection and $T^*T=1$. Hence, each finite multiplication of operators $T$ and $T^*$ coincides with
$$T^n T^{*m}=T_{n,m},\ n,m\in \mathbb{Z}_{+}.$$
Therefore finite linear combinations $\sum\limits_{n,m\in \mathbb{Z}_{+}} \lambda_{n,m}T_{n,m},\ \lambda_{n,m}\in \mathbb{C} $ form a dense subset in the Toeplitz algebra $\mathcal{T}$. Obviously, 
$$T_{n,m}T_{k,l}=\left\{\begin{matrix}
T_{n+k-m,l}& \mbox{ for } k>m\\
T_{n,l+m-k},& \mbox{ for } k<m\\ T_{n,l}& \mbox{ for } k=m \end{matrix}\right. $$
Therefore the set  $\{T_{n,m}\}_{n,m\in \mathbb{Z}_{+}}$ is a semigroup. Since $T_{n,m} T_{m,n} T_{n,m}=T_{n,m}$ for all $n,m\in \mathbb{Z}_{+}$, $\{T_{n,m}\}_{n,m\in \mathbb{Z}_{+}}$ is an  \emph{inverse} semigroup \cite{Clifford} (\emph{generalized group} by Wagner), and   elements $T_{n,m}$ and  $T_{m,n}$ are \emph{inverse} to each other (\emph{generalized inverse} elements).
\par The set $\{T_{n,n}\}_{n\in \mathbb{Z}_{+}}$ is a commutative semigroup of idempotents. Indeed, operator $T_{n,n}$ is a projection onto the subspace $$l^2(\mathbb{Z}_{+}+n)=\{f\in l^2(\mathbb{Z}_{+}):\ f(m)=0,\mbox{ for } m<n\}$$
\par Each operator $T_{n,m}$ is Fredholm with index $m-n=\mathrm{ind}\ T_{n,m}$. 
If $k=\mathrm{ind}\ T_{n,m}$ then $$T_{n,m}=\left\{\begin{matrix}
T_{n,n}T_{0,k}& \mbox{ if } k\geq0\\
T_{-k,0}T_{m,m}& \mbox{ if } k<0. \end{matrix}\right.$$

\par Let $C(S^1,\mathcal{T})$ be the C*-algebra of all continuous functions on the unit circle with values in the algebra $\mathcal{T}$, endowed with uniform norm:
$$\left\|A\right\|=\sup_{e^{i\theta}\in S^1}\left\|A(e^{i\theta})\right\|, \mbox{  }A\in C(S^1,\mathcal{T}).$$

 Each function $\widetilde{A}\in C(S^1,\mathcal{T})$ can be written in a formal Fourier series 
$$A(e^{i\theta})\simeq\sum_{n=-\infty}^{\infty}e^{in\theta}A_n, $$
$$with\ A_n=\frac{1}{2\pi}\int_0^{2\pi}A(e^{i\theta})e^{-in\theta}d\theta \in \mathcal{T}.$$
And each element $\widetilde{A}\in C(S^1,\mathcal{T})$ can be approximated by finite linear combinations $\sum A_n e^{in\theta}$ in the norm of algebra $C(S^1,\mathcal{T})$.
\par Denote by $\widetilde{\mathcal{T}}$ the closed subalgebra of algebra $C(S^1,\mathcal{T})$ generated by $\mathcal{T}$-valued functions $\widetilde{T}_{n,m}$,
$$\widetilde{T}_{n,m}(e^{i\theta})=e^{ik\theta}T_{n,m},\ with\ k=\mathrm{ind}\ T_{n,m}.$$

Note that $\{\widetilde{T}_{n,m}\}_{n,m\in\mathbb{Z}_{+}}$ is an inverse semigroup in $C(S^1,\mathcal{T})$ as  $\{T_{n,m}\}_{n,m\in\mathbb{Z}_{+}}$ is an inverse semigroup in $\mathcal{T}$. And $$\widetilde{T}_{n,m}\widetilde{T}_{k,l}=e^{i(l-n+m-k)}T_{n,m}T_{k,l},$$ with $l-n+m-k$ equal to the index of $T_{n,m}T_{k,l}$.
\par Denote by $\mathcal{T}_k$ the closed subspace in $\mathcal{T}$, generated by all linear combinations of operators $T_{n,m}$ with  $\mathrm{ind}\ T_{n,m}=k$. Then corresponding functions $\widetilde{T}_{n,m}$ generate a closed subspace $\widetilde{\mathcal{T}}_k$ in $\widetilde{\mathcal{T}}$. Formally $\widetilde{\mathcal{T}}_k=e^{ik\theta}\mathcal{T}_k$.

\par Since $\mathrm{ind}(T_{n,m}\cdot T_{k,l})=\mathrm{ind}\ T_{n,m}+\mathrm{ind}\ T_{k,l}$ we obtain $\mathcal{T}_k\cdot \mathcal{T}_l\subset\mathcal{T}_{k+l}$. Hence algebras $\mathcal{T}$ and $\widetilde{\mathcal{T}}$ are $\mathbb{Z}$-graded and can be formally written in the following way.
$$\mathcal{T}=\bigoplus_{k=-\infty}^\infty \mathcal{T}_k,\ \widetilde{\mathcal{T}}=\bigoplus_{k=-\infty}^\infty \widetilde{\mathcal{T}}_k.$$ 
The following relation holds for subspace $\mathcal{T}_k$:
\begin{equation}\label{tk}\mathcal{T}_k=\left\{\begin{matrix}
\mathcal{T}_0\cdot T_{0,k},& \mbox{ если } k>0\\
T_{-k,0}\cdot \mathcal{T}_0,& \mbox{ если } k<0 \end{matrix}\right.  \end{equation}
Therefore we will consider the structure of $\mathcal{T}_0$ foremost.

\par  $\mathcal{T}_0$ is a commutative Banach algebra, generated by the set of commutative projections $\{T_{n,n},\ n\in \mathbb{Z}_+\}$. The orthonormal basis $\{e_n\}_{n=0}^\infty$  of $l^2(\mathbb{Z}_{+})$ forms the set of eigenfunctions for operators $T_{n,n},\ n\in \mathbb{Z}_+$:
$$T_{n,n}e_k=\lambda_{k,n}e_k,$$
$$with\  \lambda_{k,n}=\left\{\begin{matrix} 0,& if\ k<n\\ 1,& if\ k\geq n. \end{matrix}\right.$$
Therefore, if $A\in\mathcal{T}_0$, then   $Ae_n=\lambda_n(A)e_n$, with $\lambda_n(A)\in\mathbb{C}$. Hence, to each $A\in \mathcal{T}_0$ one can assign   function $\widehat{A}$ on $\mathbb{Z}_+$:
$$\widehat{A}(n)=\lambda_n(A).$$

If $A=\sum\limits_{k=0}^n \beta_k T_{k,k}$ then the corresponding function is
$$\widehat{A}(m)=\lambda_m(A)=\left\{\begin{matrix} 
\sum\limits_{k=0}^m \beta_k& if\  m\leq n\\ 
\sum\limits_{k=0}^n \beta_k,& if\ m>n, \end{matrix}\right.$$
which is constant at infinity. Therefore this function belongs to $C_0(\mathbb{Z}_+)+\mathbb{C}\cdot 1$, where $C_0(\mathbb{Z}_+)$ is the space of  functions that tend to zero at infinity. Since
$$\left\|A\right\|=\left\|\widehat{A}\right\|=\sup_{n\in\mathbb{Z}_+}\left|\widehat{A}(n)\right|,$$
and algebra $C_0(\mathbb{Z}_+)+\mathbb{C}\cdot 1$ is closed with respect to this norm, then the map $A\to\widehat{A}$ is an isometric map  $\mathcal{T}_0\to C_0(\mathbb{Z}_+)+\mathbb{C}\cdot 1$. Hence, according to Gelfand-Naimark Theorem, we may assume that $$\mathcal{T}_0=C_0(\mathbb{Z}_+)+\mathbb{C}\cdot 1.$$

\section{The structures of $\mathcal{T}$ and $\widetilde{\mathcal{T}}$}\label{sec:str}
\begin{lemma}\label{fur} Each element $\widetilde{A}\in \widetilde{\mathcal{T}}$ can be written in a formal Fourier series
\begin{equation}\label{atfur}\widetilde{A}(e^{i\theta})\simeq\sum_{k=-\infty}^\infty A_k e^{ik\theta},\end{equation}
$$where\ A_k=\frac{1}{2\pi}\int_0^{2\pi}\widetilde{A}(e^{i\theta})e^{-ik\theta}d\theta \ lies\ in \  \mathcal{T}_k.$$\end{lemma}
\begin{proof} (\ref{atfur}) follows from the general theory. Thus it is sufficient to show that each  $A_k$ lies in $\mathcal{T}_k$. 

\par For every $\varepsilon>0$ there exists such $P=\sum\limits_{n,m}c_{n,m}\widetilde{T}_{n,m}$ (where the sum is finite), that
$\left\|\widetilde{A}-P\right\|<\varepsilon$. 
\par Calculate value of $P$ on $e^{i\theta}$:
$$P(e^{i\theta})=\sum_{n,m}c_{n,m}\widetilde{T}_{n,m}(e^{i\theta})=\sum_{n,m}c_{n,m}e^{i\theta(m-n)}T_{n,m}.$$
Therefore, $P$ can be written in the form of finite sum
$$P(e^{i\theta})=\sum_k e^{ik\theta} P_k,$$
where each $P_k$ lies in $\mathcal{T}_k$, i.e. it is a finite linear combination of operators with index $k$. $P_k$ is a Fourier coefficient of $P$.

$$P_k=\frac{1}{2\pi}\int_0^{2\pi}P(e^{i\theta})e^{-ik\theta}d\theta.$$
It follows that
$$\left\|A_k-P_k\right\|=\left\|\frac{1}{2\pi}\int_0^{2\pi}(\widetilde{A}(e^{i\theta})-P(e^{i\theta}))e^{-ik\theta}d\theta\right\|\leq\left\|\widetilde{A}-P\right\|<\varepsilon.$$
Since $\mathcal{T}_k$ is a closed subspace in $\mathcal{T}$ we see that $A_k$ lies in $\mathcal{T}_k$.\end{proof}

\begin{lemma}\label{izom} Algebras $\widetilde{\mathcal{T}}$ and $\mathcal{T}$ are isomorphic.\end{lemma}
\begin{proof} Consider the map $\pi\colon \widetilde{\mathcal{T}}\to \mathcal{T}$ defined by  
$$\pi(A)=A(e^{0i})=A(1)$$
Obviously, $\pi$ is a linear multiplicative map and $(\pi(A))^*=\pi(A^*)$. Futher,
$$\left\|\pi(A)\right\|=\left\|A(1)\right\| \leq\left\|A\right\|$$ Hence $\pi$ is a *-homomorphizm. Clearly,
$\pi(\widetilde{T}_{n,m})=T_{n,m}$. Since the subalgebra generated by  $T_{n,m}$ is dense in $\mathcal{T}$ and  $\pi(\widetilde{\mathcal{T}})$ is closed, we have $\pi(\widetilde{\mathcal{T}})=\mathcal{T}$. Thus $\pi$ is surjective. It remains to check that $\pi$ is injective.
\par Suppose that $\pi(A)=0,\ A\in \widetilde{\mathcal{T}}$. Then for every $\varepsilon>0$ there exists such $P=\sum\limits_{n,m}c_{n,m}\widetilde{T}_{n,m}$ (where the sum is finite), that
$\left\|A-P\right\|<\varepsilon$. Hence, $\left\|\pi P\right\|<\varepsilon$. 
Then, $$P=\sum_k e^{ik\theta} P_k,$$
where each $P_k$ lies in $\mathcal{T}_k$. Therefore $$\pi (P)=\sum_k P_k.$$ 
As in the proof of Lemma \ref{fur}, we obtain that $$\left\|A_k-P_k\right\|\leq\left\|(A-P)(e^{i\theta})\right\|\leq \left\|A-P\right\|,$$
\begin{equation}\label{eps}\left\|A_k-P_k\right\|<\varepsilon \ for\ all\ k.\end{equation}
\par  Assume that $j>0$. Hence $P_{j}\cdot T_{j,0}=P_0'$, where $P_0'\in \mathcal{T}_0$, i.e. it is a finite linear combination of idempotents (projections $T_{n,n}$ with finite rank). Therefore, $P_0'$ is a diagonal operator in the basis $\{e_n\}_{n=0}^\infty$:
$$P_0'e_n=\lambda_n e_n \ for\ all\ n,\ \lambda_n\in\mathbb{C}.$$
Hence, $\left\|P_0'\right\|$ is equal to the maximum of modules of all eigenvalues.
$$\left\|P_0'\right\|=\max_{n}\left|\lambda_n\right|$$

Let $l_j$ be a number such that $\left\|P_0'\right\|=\left|\lambda_{l_j}\right|$. Then $$\left\|P_0'e_{l_j}\right\|=\left\|P_0'\right\|.$$
Moreover, by virtue of $P_j\cdot T_{j,0}=P_0'$, we have $$\left\|P_0'\right\|\leq\left\|P_j\right\|\cdot \left\|T_{j,0}\right\|=\left\|P_j\right\|.$$
On the other hand $P_j=P_0'\cdot T_{0,j}$, so $\left\|P_j\right\|\leq\left\|P_0'\right\|$. Therefore, 
$$\left\|P_j\right\|=\left\|P_0'\right\|.$$
Hence, for every $j>0$ there exists $l_j$ such that $\left\|P_j\right\|=\left\|P_0'e_{l_k}\right\|$. Thus we have
$$\varepsilon>\left\|\pi (P)\right\|=\left\|\sum_k P_k\right\|=\left\|\sum_k P_k\right\|\cdot \left\|T_{j,0}\right\|\geq \left\|\sum_k P_k\cdot T_{j,0}\right\|\geq $$ $$\geq\left\|\sum_k P_kT_{j,0}e_{l_j}\right\|\geq $$ $$\geq \left\|P_jT_{j,0}e_{l_j}\right\| =\left\|P_0'e_{l_j}\right\|=\left|\lambda_{l_j}\right|=\left\|P_0'\right\|=\left\|P_j\right\|.$$
\par In case of $j<0$ we have $P_j=T_{-j,0}\cdot P_0',\ T_{0,-j}\cdot P_j=P_0'$ for a certain $P_0'\in \mathcal{T}_0$. And similarly we obtain $\left\|P_j\right\|<\varepsilon$. If $j=0$, operator $P_j$ is a diagonal projection itself.
Thus for every number $k$ we have $\left\|P_k\right\|<\varepsilon$. If we combine this with (\ref{eps}), we get $$\left\|A_k\right\|\leq\left\|A_k-P_k\right\|+\left\|P_k\right\|<2\varepsilon.$$ Therefore, $A_k=0$ and $A=0$.
 \end{proof}
\par  As in the theory of Fourier series  combining these two Lemmas we obtain the following statement.
\begin{theorem}\label{ab} Let $A,B$ be the elements of $\widetilde{\mathcal{T}}$, and 
$$A(e^{i\theta})\simeq\sum_{k=-\infty}^\infty A_k e^{ik\theta},\ B(e^{i\theta})\simeq\sum_{k=-\infty}^\infty B_k e^{ik\theta}.$$   
Then	 \begin{equation}1)\ \left\|A_k\right\|\leq\left\|A\right\|\ for\ all\ k.\end{equation}
	 \begin{equation}2)\ If\  AB(e^{i\theta})\simeq\sum_{k=-\infty}^\infty C_k e^{ik\theta}, \ then\ C_k=\sum_{n+m=k} A_nB_m.\end{equation}

\end{theorem}
\begin{lemma}\label{t0} Every element $A\in \mathcal{T}_0$ is uniquely represented by infinite series  $$A=\sum_{n=0}^\infty \beta_nT_{n,n},$$
which converges in the strong operator topology (but not uniformly) and $\sum\limits_{n=0}^\infty \beta_n$ converges.
\end{lemma}
\begin{proof} Let $P_n=T_{n,n}-T_{n+1,n+1}$. This operator is a projection onto linear subspace spanned by $e_n$. Then $A$ can be written in the following way 
$$A=\sum_{n=0}^\infty \alpha_n P_n,\ where\ \alpha_n=(Ae_n,e_n).$$
$\alpha_n$ is the value of corresponding function in $C_0(\mathbb{Z}_+)+\mathbb{C}\cdot 1$. It implies
$$A=\sum_{n=0}^\infty\alpha_n(T_{n,n}-T_{n+1,n+1}).$$
Let $\beta_0=\alpha_0$ and $\beta_{n+1}=\alpha_{n+1}-\alpha_n$, then $A$ is uniquely represented by $$A=\sum_{n=0}^\infty\beta_nT_{n,n}.$$
$\sum\limits_{n=0}^\infty \beta_n$ converges, because the limit of $\alpha_n$ exists.
\end{proof}

\begin{theorem}\label{afur}
Every element $A\in\mathcal{T}$ is represented by infinite series
\begin{equation}\label{atk} A=\sum_{k=-\infty}^0 T_{-k,0}\cdot (\sum_{n=0}^\infty \beta_{k,n}^A T_{n,n})+\sum_{k=1}^\infty(\sum_{n=0}^\infty \beta_{k,n}^A T_{n,n})T_{0,k},\end{equation} which converges in the strong operator topology.
\end{theorem}
\begin{proof}  Denote by $\widetilde{A}(e^{i\theta})$ the corresponding function of $A$ in $\widetilde{\mathcal{T}}$. Using Lemma \ref{fur}, we have that $\widetilde{A}$ can be written in a formal Fourier series
$$\widetilde{A}(e^{i\theta})\simeq\sum_{k=-\infty}^\infty A_ke^{ik\theta}, \ where\ A_k\in \mathcal{T}_k$$
Therefore $A$ is represented by
$$A=\sum_{k=-\infty}^\infty A_k.$$
$A_k$ lies in $\mathcal{T}_k$, and combining (\ref{tk}) and Lemma \ref{t0}, we obtain
$$A_k=(\sum_{n=0}^\infty \beta_{k,n}^A T_{n,n})\cdot T_{0,k} \ if\ k>0$$
$$and\  A_k=T_{-k,0}\cdot(\sum_{n=0}^\infty \beta_{k,n}^A T_{n,n})\ if\  k\leq0.$$
Thus, we have that $A$ is represented by (\ref{atk}).   \end{proof}

\begin{definition}In what follows $A_k$ is called a coefficient of Fourier series for $\widetilde{A}$.\end{definition}
\section{The compact operators}\label{sec:comp}
\par
Since $\widetilde{T}_{n,m}(e^{i\theta})=T_{n,m}e^{ik\theta}$, where $k=\mathrm{ind}\ T_{n,m}$, then $$\widetilde{T}_{n,m}(e^{i(\theta+\theta_0)})=e^{ik\theta_0}\widetilde{T}_{n,m}(e^{i\theta}).$$ Hence, algebra $\widetilde{\mathcal{T}}$ is invariant under shifts by elements of group $S^1$. Therefore, $S^1$ is isomorphic to a certain subgroup of the group of automorphisms $\mathrm{Aut} \widetilde{\mathcal{T}}$. Let $\alpha_{\theta_0}$ be the rotation of $S^1$ by the angle $\theta_0$, $0\leq\theta_0\leq2\pi$. It follows that the coefficients $A_k$ of Fourier series for $A\in\widetilde{\mathcal{T}}$ can be written in the following way:
\begin{equation}\label{aka}A_k=\frac{1}{2\pi}\int_0^{2\pi}\alpha_\theta(A)e^{-ik\theta}d\theta.\end{equation}

\begin{lemma}\label{a0}
Operator
\begin{equation}\label{abn}A=\sum_{n=0}^\infty \beta_n^A T_{n,n}\end{equation}
in $\mathcal{T}_0$ is compact iff
\begin{equation}\label{b0} \sum_{n=0}^\infty \beta_n^A=0.\end{equation}
\end{lemma}
\begin{proof}Consider value of $A$ on basic element $e_n\in l^2(\mathbb{Z}_{+})$. If $A$ is represented by (\ref{abn}), we have
$$Ae_n=(\sum_{k=0}^n \beta_k^A)e_n.$$
This means that $A$ is diagonal with respect to the basis $\{e_n\}_{n=0}^\infty$. It follows that $A$ is compact iff 
$$\lim_{n\to\infty}\sum_{k=0}^n \beta_k^A=0. $$
Thus, we obtain (\ref{b0}).
   \end{proof}
\begin{theorem}\label{comp}
Suppose $A\in \mathcal{T}$ is represented by series (\ref{atk}). Then $A$ is compact iff 
\begin{equation}\label{bkn0}\sum_{n=0}^\infty \beta_{k,n}^A=0\end{equation}
for all numbers $k$.
\end{theorem}

\begin{proof} Consider compact operator  $A\in \mathcal{T}$ with series $$A=\sum_{k=-\infty}^\infty A_k,\ A_k\in\mathcal{T}_k.$$  
Let us show that $A_k$ is compact for every $k$. Since $\alpha_\theta\colon \mathcal{T}\to \mathcal{T}$ is a *-automorphism, then  operator $\alpha_\theta(A)$ is compact for every $\theta\in[0,2\pi]$. We recall that $A_k$ can be written in a form of integral (\ref{aka}). This integral is a uniform limit of the Riemann sums, which are linear combinations of compact operators of the form $\alpha_\theta(A)e^{-ik\theta}$. Thus, $A_k$ is compact.
\par Let $k$ be a negative number. Then $$A_k=T_{-k,0}\cdot\sum_{n=0}^\infty \beta_{k,n}^A T_{n,n}.$$
This implies the compactness of the following operator.
$$T_{0,-k}A_k=\sum_{n=0}^\infty \beta_{k,n}^A T_{n,n}\in \mathcal{T}_0.$$
By Lemma \ref{a0}, we obtain (\ref{bkn0}).
If $k>0$, then $$A_k=(\sum_{n=0}^\infty \beta_{k,n}^A T_{n,n})T_{0,k}.$$
Therefore $$A_kT_{k,0}=\sum_{n=0}^\infty \beta_{k,n}^A T_{n,n}$$
is compact. Hence, we have again (\ref{bkn0}).
\par We now prove the converse. Suppose operator $A\in \mathcal{T}$ represented by the series (\ref{atk}) satisfies (\ref{bkn0}) for all numbers $k$. By Lemma \ref{a0} we have that each  $A_k$ is a compact operator. Let $\widetilde{A}$ be the corresponding function in $\widetilde{\mathcal{T}}$ for $A$ with the following Fourier series.
$$\widetilde{A}(e^{i\theta})\simeq\sum_{k=-\infty}^\infty e^{ik\theta}A_k,\ A_k\in\mathcal{T}_k,\ 0\leq\theta\leq2\pi.$$
Denote by $S_n(e^{i\theta})$ the partial sum of the Fourier series for $\widetilde{A}$.
$$S_n(e^{i\theta})=\sum_{k=-n}^n e^{ik\theta}A_k$$
By the general theory (\cite{Gofman}), $\widetilde{A}(e^{i\theta})$ can be approximated in the norm by Ces\'aro means of the partial sums:
$$\sigma_n(e^{i\theta})=\frac{1}{n}(S_0(e^{i\theta})+S_1(e^{i\theta})+\ldots+ S_{n-1}(e^{i\theta}))$$
Therefore $A=\widetilde{A}(1)$ can be approximated by operators $\sigma_n(1)$ in the supremum norm. Operator $\sigma_n(1)$ is a finite sum of compact operators $A_k$ and hence compact. Thus, $A$ is a compact operator.\end{proof}

\section{The compact quantum semigroup}\label{sec:q}

\begin{definition}
Let $\mathcal{A}$ be a unital $C^*$-algebra and $\mathcal{A}\otimes \mathcal{A}$ be the minimal $C^*$-tensor product of $\mathcal{A}$ onto itself. And let $\Delta$ be a unital *-homomorphism $\Delta\colon \mathcal{A}\to \mathcal{A}\otimes \mathcal{A}$ which satisfies the condition of coassociativity:
\begin{equation}\label{coass}(\Delta\otimes\mathrm{id})\Delta=(\mathrm{id}\otimes\Delta)\Delta. \end{equation}
Then $(\mathcal{A},\Delta)$ is called a compact quantum semigroup \cite{Sadr}, the map $\Delta$ is called a comultiplication.
\end{definition}
Let us show that there exists a structure of compact quantum semigroup on algebra $\mathcal{T}$. Define comultiplication $\Delta\colon \mathcal{T}\to \mathcal{T}\otimes \mathcal{T}$. At first we define it on the left shift operator $T$:
\begin{equation}\label{dt}\Delta(T)=T\otimes T. \end{equation} The next Theorem shows the way this map is extended to the entire algebra $\mathcal{T}$.
\begin{theorem}\label{cqt}$(\mathcal{T},\Delta)$ is a compact quantum semigroup. \end{theorem}
\begin{proof}
The map $\Delta$, given by (\ref{dt}), is naturally extended to the algebra, generated by operator $T$:
$$\Delta(\alpha T)=\alpha T\otimes T,\ \Delta(T^*)=T^*\otimes T^*= (T\otimes T)^*,\ \Delta(T_{n,m})=T_{n,m}\otimes T_{n,m}$$
Let $A,B$ be finite linear combinations of operators $T_{n,m}$: $$A=\sum_{n,m}\alpha_{n,m}T_{n,m},\ B=\sum_{k,l}\beta_{k,l}T_{k,l}.$$ Then we have $$ \Delta(A\cdot B)=\sum_{n,m}\sum_{k,l}\alpha_{n,m}\beta_{k,l}T_{n,m}T_{k,l}\otimes T_{n,m}T_{k,l}=\Delta(A)\cdot \Delta(B).$$
Therefore $\Delta$ is an involutive multiplicative linear map on the dense *-subalgebra in $\mathcal{T}$. It is clear that $\Delta$ is a unital homomorphism.
\par Now, by Coburn's Theorem \cite{Murphy} there exists a unique unital *-homomorphism of $\mathcal{T}$ to $\mathcal{T}\otimes \mathcal{T}$, satisfying relation (\ref{dt}). Hence, $\Delta$ is a unital *-homomorphism on $\mathcal{T}$.
\end{proof}

\begin{definition}  Compact quantum semigroup  $(\mathcal{A},\Delta)$ is called cocommutative if  
$$\Delta(a)=\sigma(\Delta(a)),\ a\in \mathcal{A}$$
where $\sigma\colon \mathcal{A}\otimes \mathcal{A}\to \mathcal{A}\otimes \mathcal{A}$ is the flip-automorphism given by $\sigma(a\otimes b)=b\otimes a$  for all $a,b\in \mathcal{A}$.
\end{definition}

\begin{remark}\label{comm} Obviously, compact quantum semigroup $(\mathcal{T},\Delta)$ is cocommutative. Moreover, by the above-mentioned Coburn's Theorem we have that $\Delta$ is isometric, i.e. is norm-preserving.
\begin{equation}\label{normd} \left\|\Delta(A)\right\|=\left\|A\right\|\ for\ all\ A\in \mathcal{T}.\end{equation}
 \end{remark}

\begin{definition}\cite{Woronowicz} A compact quantum semigroup $(\mathcal{A},\Delta)$ is called compact quantum group if the linear subspaces \begin{equation}\label{w1}\{(a\otimes I)\Delta(b)\colon a,b\in \mathcal{A}\} \end{equation}
 \begin{equation}\label{w2}\{(I\otimes a)\Delta(b)\colon a,b\in \mathcal{A}\}  \end{equation}
are dense in $\mathcal{A}\otimes \mathcal{A}$.   \end{definition}
 \begin{theorem}\label{no} $(\mathcal{T},\Delta)$ is not a compact quantum group. \end{theorem}
 \begin{proof} Let us show that the subspace (\ref{w1}) is not dense in $\mathcal{T}\otimes \mathcal{T}$. Assume the converse. Consider operator $A=T_{1,0}\otimes T_{0,1}$ and vector $x=e_0\otimes e_1\in l^2(\mathbb{Z}_+)\otimes l^2(\mathbb{Z}_+)$. Then $A$ is approximated by a finite linear combination 
 $$B=\sum_{k,l,n,m}c_{k,l,n,m}(T_{n,m}\otimes I)(T_{k,l}\otimes T_{k,l}).$$
 Clearly, we have
 $$(T_{k,l}\otimes T_{k,l})x=\left\{\begin{matrix}
0& \mbox{ fir } l\neq0\\
e_k\otimes e_{k+1}& \mbox{ for } l=0. \end{matrix}\right.$$
 Calculate value of $B$ on vector $x$:
 $$Bx=\sum_{k,l,n,m}c_{k,l,n,m}(T_{n,m}\otimes I)(T_{k,l}\otimes T_{k,l})(e_0\otimes e_1)=$$ $$= \sum_{k,n,m} c_{k,0,n,m}(T_{n,m}\otimes I)e_k\otimes e_{k+1}=\sum_{k} x_k\otimes e_{k+1},$$
 where $x_k$ are certain vectors in $l^2(\mathbb{Z}_+)$. From the other hand,
  $$Ax=e_1\otimes e_0.$$
  Hence, $e_{k+1}=e_0$, which is nonsense. \end{proof}
 
  \begin{definition} Bialgebra $\mathcal{A}$ is called a weak Hopf algebra \cite{Li}, if there exists a linear map $S\colon \mathcal{A}\to \mathcal{A}$ satisfying the following conditions
 \begin{equation}\label{wa} \mu(\mathrm{id}\otimes S\otimes\mathrm{id})\Delta^2=\mathrm{id},\ \mu( S\otimes\mathrm{id}\otimes S)\Delta^2=S \end{equation}
 where $\Delta^2=(\Delta\otimes \mathrm{id})\Delta$, and $\mu\colon \mathcal{A}\odot \mathcal{A}\to \mathcal{A}$ is a linear associative map given by
 $$\mu(a\otimes b)=ab.$$  $S$ is called a weak antipode.
  \end{definition}
  Recall that $\{T_{n,m}\}_{n,m\in\mathbb{Z}_+}$ is an inverse semigroup. Denote by $\mathcal{P}$ the algebra, generated by all finite linear combinations of the elements from the set $\{T_{n,m}\}_{n,m\in\mathbb{Z}_+}$. Obviously, $\mathcal{P}$ is dense in $\mathcal{T}$.  Denote by $S$ the map that takes each element of this set to it's inverse element:
$$S(T_{n,m})=T_{m,n}$$ 
$S$ is naturally extended to a linear antimultiplicative map on $\mathcal{P}$. It is evident that the conditions (\ref{wa}) are satisfied for $S, \ \Delta$ on $\mathcal{P}$. We obtain the following statement.
\begin{theorem}\label{weak} The compact quantum semigroup $(\mathcal{T},\Delta)$ contains a dense weak Hopf algebra $(\mathcal{P},\Delta,S)$.\end{theorem}

\section{The Dual Algebra and the Haar Functional}\label{sec:dual}
\par Denote by $\mathcal{T}^*$ the space of all linear continuous functionals on $\mathcal{T}$. Define multiplication on $\mathcal{T}^*$.
\begin{equation}\label{pr} (\rho\ast \varphi) (A)= (\rho\otimes \varphi)(\Delta(A))\end{equation}
\begin{theorem}\label{ts} $\mathcal{T}^*$ is a commutative banach algebra with multiplication given by (\ref{pr}). \end{theorem}\begin{proof} 
Indeed, it is clear that the linear space $\mathcal{T}^*$ becomes an algebra with the operation $\ast$. Moreover, $\mathcal{T}^*$ is endowed with the  supremum norm.
$$\left\|\rho_1\ast\rho_2\right\|=\sup_{\left\|A\right\|\leq1}\left|(\rho_1\ast\rho_2)(A)\right|$$ 
It is sufficient to check the submultiplicativity of this norm on $\mathcal{T}^*$. Consider $\rho_1,\rho_2 \in \mathcal{T}^*$. If we combine relations (\ref{normd}), (\ref{pr}) with the properties of the functionals tensor product, we get
$$\left\|\rho_1\ast\rho_2\right\|=\sup_{\left\|A\right\|\leq1}\left|(\rho_1\ast\rho_2)(A)\right|=\sup_{\left\|A\right\|\leq1}\left|(\rho_1\otimes\rho_2)\Delta(A)\right|=$$ $$=\sup_{\left\|\Delta(A)\right\|\leq1}\left|(\rho_1\otimes\rho_2)\Delta(A)\right|\leq \sup_{B\in\mathcal{T}\otimes \mathcal{T},\  \left\|B\right\|\leq1}\left|(\rho_1\otimes\rho_2)(B)\right|=\left\|\rho_1\right\|\cdot \left\|\rho_2\right\|.$$
Thus, $\mathcal{T}^*$ is the Banach algebra with the induced norm. The cocommutativeness of the compact quantum semigroup $(\mathcal{T},\Delta)$ (see Remark \ref{comm}) implies that $\mathcal{T}^*$ is commutative.
\end{proof}

\begin{lemma}\label{unital} The algebra $\mathcal{T}^*$ is unital. The unit is a functional $\epsilon$ given by \begin{equation}\label{unit}\epsilon(T_{n,m})=1\ for\ all\ n,m\in\mathbb{Z}_+.\end{equation} 
\end{lemma}
\begin{proof}
We must check that $\epsilon$ is a contionuous functional. Denote by $\mathcal{K}$ the algebra af compact operators on $l^2(\mathbb{Z}_{+})$. The quotient algebra $\mathcal{T}/\mathcal{K}$ is isometrically isomorphic to the algebra of continuous functions on the unit circle $C(S^1)$ \cite{Murphy}.  This isomorphism takes each element $T_{n,m}$ to the function $e^{ik\theta}$, where $k=\mathrm{ind}\ T_{n,m}$. Denote by $\gamma$ the composition of this isomorphism and a quotient map. $$\gamma\colon\mathcal{T}\to\mathcal{T}/K\to C(S^1)$$ Obviously, $\gamma$ is a continuous homomorphism. Consider the Dirac measure on the circle $S^1$.
$$\delta_0(f)=f(1).$$
$\delta_0$ is a continuous functional on $S^1$ with norm equal to 1, $\delta_0(e^{ik\theta})=1$ for all $k$. It is evident that $\epsilon=\delta_0\circ\gamma$. Therefore, $\epsilon$ is a continuous functional. Moreover, for any $n,m\in\mathbb{Z}_+$ $$(\epsilon\ast\rho)(T_{n,m})=\epsilon(T_{n,m})\cdot \rho(T_{n,m})=\rho(T_{n,m}).$$
This completes the proof.\end{proof}
\begin{remark}\label{counit} The functional $\epsilon$ defined by equation (\ref{unit}) is a counit \cite{Woronowicz} for the weak Hopf algebra $\mathcal{P}$. However, by Lemma \ref{unital}, we have that $\epsilon$ acts on the entire algebra $\mathcal{T}$. Therefore $\epsilon$ appears to be the counit for the compact quantum semigroup $(\mathcal{T},\Delta)$. 
\end{remark}

\begin{definition} A state $h\in\mathcal{T}^*$ is called a Haar functional \cite{Woronowicz} in $\mathcal{T}^*$ if the following conditions hold for any $\rho\in \mathcal{T}^*$:
\begin{equation}\label{haar}h\ast\rho=\rho\ast h=\lambda_\rho \cdot h,\ \lambda_\rho\in\mathbb{C}.\end{equation} 
  \end{definition}
  
\begin{theorem}\label{haart} A state $h$ is a Haar functional in $\mathcal{T}^*$ iff
\begin{equation}\label{ht}h(T_{n,m})=\left\{\begin{matrix}
1,& \ if\ n=m=0\\
0,& \ if\ either\ n\neq 0 \ or\ m\neq 0 \end{matrix}\right. \end{equation}
 \end{theorem}
\begin{proof}
Let us show that the Haar functional $h$ satisfies (\ref{ht}). Since $h$ is a state, then $h(T_{0,0})=1$. Suppose there exist numbers $n,m\in\mathbb{Z}_+$ such that $h(T_{n,m})\neq 0$. Then, by equations (\ref{haar}), for any functional $\rho$ we have
$$(h\ast\rho)(T_{n,m})=h(T_{n,m})\rho(T_{n,m})=\lambda_\rho h(T_{n,m}).$$
It means that $\rho(T_{n,m})=\lambda_\rho$. On the other hand, $\rho(I)=\lambda_\rho$. We have $\rho(T_{n,m})=\rho(I)$ for any functional $\rho$. But it is clear that there exists a functional with different values on $T_{n,m}$ and $I$. Hence, $h(T_{n,m})= 0$.
\par Now assume that $h$ is a functional which satisfies (\ref{ht}). It is easily proved that (\ref{haar}) holds for any functional $\rho$ on any operator $T_{n,m}$. To conclude the proof, it remains to note that for any $A\in \mathcal{T}$,
$$h(A)=(Ae_0,e_0).$$
Therefore, $h$ is linear and continuous on the entire algebra $\mathcal{T}$.\end{proof}

The next remark shows that the Haar functional in $\mathcal{T}^*$ can be obtained in the same way as in \cite{Woronowicz}, despite the fact that $(\mathcal{T},\Delta)$ is not a compact quantum group. 
\begin{remark}\label{cez}
Suppose $\rho$ is a faithful state on $\mathcal{T}$. Denote $\rho_n=\frac{1}{n}\sum\limits_{k=1}^n \rho^k$. Then $\lim\limits_{n\to\infty}\rho_n=h$.
\end{remark}
\begin{proof}
Let us show that $\left|\rho(T_{n,m})\right|<1$ if either $n\neq 0$ or $m\neq 0$. Assume that $\left|\rho(T_{n,m})\right|=1$. Then, using the properties of linear functionals on the $C^*$-algebras, we get
$$\rho(T_{n,m})\rho(T_{n,m}^*)\leq\rho(T_{n,m}^*T_{n,m}) $$ 
Therefore, $\rho(T_{m,m})=1$. This means that $\rho(T_{0,0}-T_{m,m})=0$. But the element $T_{0,0}-T_{m,m}$ is positive. Hence, $\rho$ is not faithful, which contradicts with the assumption. This implies 
$$\lim_{k\to\infty}\rho^k(T_{n,m})=0$$
for any numbers $n,m$ such that $n\neq 0$ or $m\neq 0$.
Thus, we have 
$$\lim_{k\to\infty}\rho_k(T_{n,m})= \left\{\begin{matrix}
1,& \ if\ n=m=0\\
0,& \ if\ either\ n\neq 0 \ or\ m\neq 0. \end{matrix}\right.$$
\end{proof}

\par Denote by $\mathcal{K}^\bot$ the space of functionals in $\mathcal{T}^*$, that have zero value on operators from $\mathcal{K}$:
$$\mathcal{K}^\bot=\{\rho\in\mathcal{T}^*\mid\ \rho|_{\mathcal{K}}=0\}.$$
\begin{lemma}\label{rot0} Suppose $\rho\in\mathcal{T}^*$. Then the next conditions are equivalent:
\begin{enumerate}
	\item $\rho\in \mathcal{K}^\bot$
	\item $\rho(T_{n,m})=\rho(T_{k,l})$ if $m-n=l-k$.
\end{enumerate}
\end{lemma}
\begin{proof}
Suppose $\rho\in \mathcal{K}^\bot$ and $m-n=l-k$. Then $T_{n,m}-T_{k,l}$ is a compact operator by Theorem \ref{comp}. Therefore, $\rho(T_{n,m})=\rho(T_{k,l})$.\par Now assume that $A$ is a compact operator, represented by (\ref{atk}). Take $k>0$, and denote $\rho(T_{n,n+k})=\rho_k$. Then, by Theorem \ref{comp}, we obtain
$$\rho(A_k)=\sum_{n=0}^\infty \beta_{k,n}\rho(T_{n,n+k})=\rho_k\sum_{n=0}^\infty \beta_{k,n}=0.$$
Similarly, if $k\leq0$, then $\rho(A_k)=0$. Thus, $\rho(A)=0$. \end{proof}
\begin{corollary}\label{t0a} $\mathcal{K}^\bot$ is a subalgebra in $\mathcal{T}^*$.
\end{corollary}\begin{proof} Suppose $\rho,\varphi\in\mathcal{K}^\bot$. Let us show that the condition (2) of Lemma \ref{rot0} holds for $\rho\ast\varphi$. Take $m,n,l,k$ such that  $m-n=l-k$. Then we have
$$(\rho\ast\varphi)(T_{n,m})=\rho(T_{n,m})\varphi(T_{n,m})=\rho(T_{k,l})\varphi(T_{k,l})=(\rho\ast\varphi)(T_{k,l}).$$
\end{proof}

\begin{theorem}\label{t0ms} The algebra $\mathcal{K}^\bot$ is isomorphic to the algebra $M(S^1)$ of regular Borel measures on the circle. Moreover, the multiplication $\ast$, given by (\ref{pr}) coincides with the convolution in $M(S^1)$. \end{theorem}
\begin{proof}
Let us show that the algebra $(\mathcal{T}/K)^*$ is isometrically isomorphic to $\mathcal{K}^\bot$. Indeed, let $\pi_1\colon\mathcal{T}\to\mathcal{T}/K$ be a canonical quotient homomorphism. $\pi_1$ induces the dual map $\pi_1^*\colon (\mathcal{T}/K)^*\to \mathcal{T}^*$, given by 
$$\pi_1^*(\varphi)(A)=\varphi(\pi_1(A)) \ for\ any\  \varphi\in (\mathcal{T}/K)^*,\ A\in \mathcal{T}.$$ 
Take a compact operator $A$ and $\varphi\in (\mathcal{T}/K)^*$. Then $\pi_1^*(\varphi)(A)=\varphi(\pi_1(A))=\varphi(0)=0$. Therefore, $\pi_1^*(\varphi)\in \mathcal{K}^\bot$. This means that $\pi_1^*\colon (\mathcal{T}/K)^*\to \mathcal{K}^\bot$.
\par Suppose $\rho\in \mathcal{K}^\bot$. For any equivalence class $[A]$ with $A\in \mathcal{T}$ define a functional $\rho'\in (\mathcal{T}/K)^*$ by
$$\rho'([A])=\rho(A).$$
 Since $\rho'([A])=\rho(A)=0$ for any compact operator $A$, the linear functional $\rho'$ is well-defined. Obviously, $\pi_1^*(\rho')=\rho$, 
i.e.  $\pi_1^*$ is a bijection of $(\mathcal{T}/K)^*$ on $\mathcal{K}^\bot$. Furthermore, $\pi_1^*$ is norm-preserving.
 \par We recall that $\mathcal{T}/K$ is isometrically isomorphic to the algebra $C(S^1)$ of continuous functions on the unit circle \cite{Murphy}.  This isomorphism takes each operator $T_{n,m}$ to the function $e^{ik\theta}$, where $k=\mathrm{ind}\ T_{n,m}$. Denote this isomorphism by $\pi_2\colon \mathcal{T}/K\to C(S^1)$, and the dual map $\pi_2^*\colon (C(S^1))^*\to (\mathcal{T}/K)^*$. $\pi_2^*$ is a bijection also. By Riesz theorem, for any continuous linear functional $\rho$ on $C(S^1)$ there exists a unique regular Borel measure $\mu_\rho$ on $S^1$ such that 
 \begin{equation}\label{riss}\rho(f)=\int\limits_{S^1}f(e^{i\theta})d\mu_\rho(\theta)\ for \ any\ f\in C(S^1).\end{equation}
 It is well known that $\gamma\colon M(S^1)\to (C(S^1))^*$, which takes each regular Borel measure $\mu_\rho$ on the circle to the linear continuous functional $\rho$ on $C(S^1)$ is linear bijective and norm-preserving. Thus, $\pi_1^*\circ\pi_2^*\circ\gamma\colon M(S^1)\to \mathcal{K}^\bot$ is a norm-preserving linear bijection. 
 \par The algebra $M(S^1)$ is endowed with a convolution $(\mu_1\ast\mu_2)$ \cite{Gofman}.
 определена свертка мер $(\mu_1\ast\mu_2)$ следующим образом:
 $$\int\limits_{S^1}f(e^{i\theta})d(\mu_1\ast\mu_2)(\theta)=\iint\limits_{S^1}f(e^{i\theta_1}e^{i\theta_2})d\mu_1(\theta_1)d\mu_2(\theta_2)$$
  Let us show the multiplicativity of $\pi_1^*\circ\pi_2^*\circ\gamma$ using (\ref{pr}). By Lemma (\ref{rot0}) it is sufficient to check it for operators $T_{n,n+k},\ n,k\in\mathbb{Z}_+$. To this end, take $\rho,\varphi\in\mathcal{K}^\bot$, and the corresponding functionals $\rho_1,\varphi_1\in (\mathcal{T}/K)^*$ and $\rho_2,\varphi_2\in (C(S^1))^*$.  
    $$(\rho\ast\varphi)(T_{n,n+k})=(\rho\otimes\varphi)(\Delta(T_{n,n+k}))=(\rho\otimes\varphi)(T_{n,n+k}\otimes T_{n,n+k})=$$ $$=\rho(T_{n,n+k})\cdot\varphi( T_{n,n+k})=$$
  $$=\rho_1([T_{n,n+k}])\cdot\varphi_1( [T_{n,n+k}])=\rho_2(e^{ik\theta_1})\cdot\varphi_2( e^{ik\theta_2})=$$ $$=\int\limits_{S^1}e^{ik\theta_1}d\mu_\rho(\theta_1)\int\limits_{S^1}e^{ik\theta_2}d\mu_\varphi(\theta_2)=\iint\limits_{S^1}e^{ik\theta_1}e^{ik\theta_2}d\mu_\rho(\theta_1)d\mu_\varphi(\theta_2)=$$ $$=\int\limits_{S^1}e^{ik\theta}d(\mu_\rho\ast\mu_\varphi)(\theta)$$ 
  This completes the proof of the Theorem.
\end{proof}

\begin{definition}A state $h\in \mathcal{K}^\bot$ is called a Haar functional \cite{Woronowicz} in $\mathcal{K}^\bot$, if the conditions (\ref{haar}) hold for any functional $\rho\in \mathcal{K}^\bot$.
  \end{definition}

\begin{theorem}\label{haart0} There exists the Haar functional in $\mathcal{K}^\bot$ and it corresponds to the Haar measure on the circle.\end{theorem}
\begin{proof}
Define the functional $h_0$:
\begin{equation}\label{h0}h_0(T_{n,m})=\left\{\begin{matrix}
1,& \ if\ n=m\\
0,& \ if\ n\neq m \end{matrix}\right.\end{equation}
The condition (2) in Lemma \ref{rot0} holds for $h_0$ obviously, and the equations (\ref{haar}) hold for any $\rho\in \mathcal{K}^\bot$.  Suppose $\mu$ is the Haar measure  on the circle $S^1$, and $\hat{\mu}\colon\mathbb{Z}\to\mathbb{C}$  is the Fourier transform of $\mu$ \cite{Gofman}. 
$$\hat{\mu}(k)=\int\limits_{S^1}e^{-ik\theta}d\mu(\theta)$$
Then, by the definition of the Haar measure, we have 
$$\hat{\mu}(k)=\left\{\begin{matrix}
1,& \ if\ k=0\\
0,& \ if\ k\neq 0 \end{matrix}\right.$$
Therefore, $h_0(T_{n,n+k})=\hat{\mu}(k)$ for all $k$.  Thus, $h_0\in \mathcal{T}_0^*$ is a Haar functional in $\mathcal{K}^\bot$, corresponding to the Haar measure on the circle.\end{proof}

\end{document}